\documentclass[12pt]{amsart}
\usepackage{amscd,mathpazo,fullpage}

\usepackage{tikz-cd}												

\usepackage{booktabs}												
\newtheorem{thm}{Theorem}
\newtheorem{prop}[thm]{Proposition}
\newtheorem{lem}[thm]{Lemma}

\theoremstyle{remark}
\newtheorem{rem}[thm]{Remark}

\theoremstyle{definition}

\newtheorem{ex}[thm]{Example}

\newcommand{\C}{\mathbb{ C}}

\newcommand{\lra}{\longrightarrow}

\newcommand{\OO}{\mathcal{ O}}

\newcommand{\PP}{\mathbb{ P}}
\newcommand{\R}{\mathbb{ R}}

\newcommand{\Z}{\mathbb{ Z}}

\title{The complex geometry of two exceptional flag manifolds}
\author{D.~Kotschick}
\address{Mathematisches Institut, Ludwig-Maximilians-Universit\"at
M\"unchen, Theresienstr.~39, 80333 M\"unchen, Germany}
\email{dieter@math.lmu.de}
\author{D.~K.~Thung}
\address{Department of Mathematics, Universit\"at Hamburg, Bundesstr.~55, 20146 Hamburg, Germany}
\email{daniel.thung@uni-hamburg.de}

\date{1 October 2019, revised 19 January 2020; {\copyright \ D.~Kotschick and D.~K.~Thung 2019}}
\thanks{The second author is supported by the German Research Foundation (DFG) as part of RTG 1670.}

\begin{document}

\begin{abstract}
We discuss the complex geometry of two complex five-dimensional K\"ahler manifolds
which are homogeneous under the exceptional Lie group $G_2$. For one of these manifolds rigidity 
of the complex structure among all K\"ahlerian complex structures was proved by Brieskorn, for the 
other one we prove it here. We relate the K\"ahler assumption in Brieskorn's theorem to the question 
of existence of a complex structure on the six-dimensional sphere, and we compute the Chern numbers 
of all $G_2$-invariant almost complex structures on these manifolds. 
\end{abstract}

\maketitle

\section{Introduction}\label{s:intro}

In this paper we study the complex geometry of the two homogeneous spaces $Q$ and $Z$ appearing in the 
diagram of $G_2$-invariant fibrations displayed in Figure~\ref{diag}. They are both (co-)adjoint orbits of $G_2$,
of the form $G_2/U(2)$, 
for two non-conjugate embeddings $U(2)\hookrightarrow G_2$. These subgroups are maximally parabolic, and the 
quotients are examples of exceptional partial flag manifolds\footnote{The full flag manifold $G_2/T^2$ is discussed 
briefly in Section~\ref{s:full} below.}.

\begin{figure}[ht!]\label{diag}
	\centering
	\begin{tikzcd}[column sep=0.1cm]		
		& & G_2/T^2 \ar[dl] \ar[dr] & \\
		& Q=G_2/U(2)_- \ar[dr,"\pi_Q"'] \ar[ddl,"p" ] & 
		& G_2/U(2)_+=Z\ar[dl,"\pi_Z"]\\
		& & G_2/SO(4)=M \\[-0.75cm] 
		S^6=G_2/SU(3) & & 
	\end{tikzcd}\hspace{2cm}
	\caption{Diagram of fibrations between $G_2$-homogeneous spaces; cf.~\cite[p.~164]{SalamonBook} and~\cite{SvWood}.}
\end{figure}
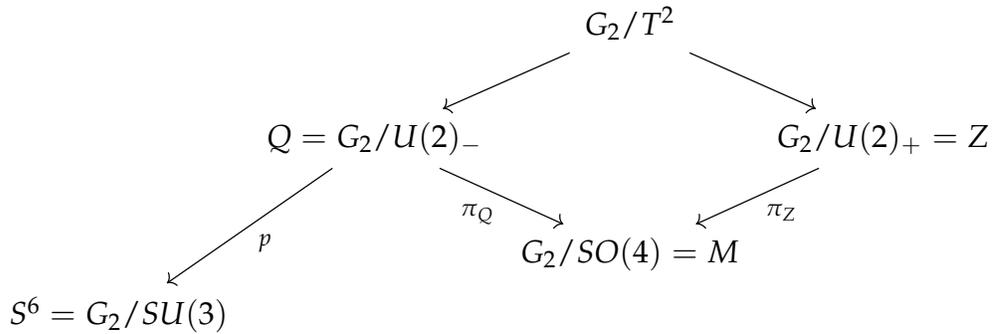

The manifold $Z$ is the Salamon~\cite{S} twistor space of the exceptional Wolf~\cite{Wolf} space $M=G_2/SO(4)$ considered 
as a quaternionic K\"ahler manifold of positive scalar curvature. As such it has the structure of a smooth Fano variety, and it 
carries a holomorphic contact structure. The other quotient of $G_2$ by $U(2)$ is denoted by $Q$ because it is diffeomorphic 
to a smooth quadric hypersurface in $\C P^6$. Thus it also carries the structure of a smooth Fano variety. Indeed the complex 
structures are $G_2$-invariant and there is a unique invariant K\"ahler--Einstein metric of positive scalar curvature in both cases.
The distinction between $U(2)_-$ and $U(2)_+$ is best described in terms of octonions, as in~\cite{Bryant,Kerr,SvWood}.
Without getting involved in the details, one can always distinguish $Q$ and $Z$ by remembering that the isotropy 
representation of $Q$ splits into three irreducible summands, whereas the isotropy representation of $Z$ has only two summands.

\subsection{Rigidity of standard complex structures}

It is a classical result of Hirzebruch--Ko\-dai\-ra~\cite{HK} and Yau~\cite{Yau} that on the manifold underlying complex projective space
the standard structure is the unique K\"ahlerian complex structure. Since~\cite{HK}, such rigidity results have 
been proved for a few other manifolds, for example for the odd-dimensional quadrics by Brieskorn~\cite{Brieskorn}. Like the result of 
Hirzebruch and Kodaira, many of these extensions depend on the fact that they consider manifolds with very simple cohomology 
algebras. We refer the reader to~\cite{LW,Li,Debarre} for accounts of some refined results in the spirit of~\cite{HK}.
As explained in~\cite{Debarre} and the references given there, any compact K\"ahler manifold with the integral cohomology
ring of $\C P^5$ is biholomorphic to it. The manifolds $Q$ and $Z$ show that this fomulation is sharp. They are simply connected 
compact oriented $10$-manifolds with the same homology and cohomology {\it groups} as $\C P^5$, but with different and distinct 
{\it ring structures} on cohomology\footnote{Note however that $\C P^5$ is spin, whereas $Q$ and $Z$ are not.}.

Brieskorn's result~\cite{Brieskorn} shows that the manifold $Q$ has a unique K\"ahlerian complex structure, 
without any assumption about it being homogeneous or Fano. We prove below the analogous statement for the manifold $Z$.
\begin{thm}\label{Zrigid}
Any K\"ahlerian complex manifold homeomorphic to the twistor space $Z$ is biholomorphic to it.
\end{thm}
Note that we consider all possible complex structures within the homeomorphism type of $Z$, assuming only that they admit some 
K\"ahler metric. We do not assume that the structure is Fano, or that it admits a holomorphic contact structure.
These properties will follow {\it a posteriori} from the proof. For this particular manifold, Theorem~\ref{Zrigid} improves a partial 
result of Hwang~\cite{Hwang} for arbitrary homogeneous Fano contact manifolds with $b_2=1$.

The K\"ahler assumption in Theorem~\ref{Zrigid} is crucial, and we do not know whether the result holds without it.
In this spirit, it is well known that if uniqueness of the complex structure on $\C P^3$ could be proved without the K\"ahler assumption,
then it would follow that $S^6$ cannot have a complex structure\footnote{Although both existence and non-existence of a complex 
structure on $S^6$ has been claimed many times over the years, this issue seems to be still unresolved.}; 
cf.~\cite{Hir1,KSurvey}. This is because the blowup at a point of a complex $S^6$ yields
a non-K\"ahler complex structure on $\C P^3$. There is a similar relation between potential complex structures on $S^6$ and 
non-K\"ahler complex structures on the five-dimensional quadric $Q$, which seems not to have been noticed before.

\begin{thm}\label{Qproj}
If $S^6$ admits a complex structure, then the manifold $Q$ admits two distinct non-K\"ahler complex structures, at least one of 
which carries a holomorphic contact structure.
\end{thm}
The complex structures being distinct means that they are not equivalent under the equivalence relation generated by conjugation,
diffeomorphism, and homotopies of almost complex structures. The two complex structures arise from the projectivized tangent and 
cotangent bundles of the putative complex structure on $S^6$. They are non-K\"ahler because $S^6$ cannot be K\"ahler, 
or by Brieskorn's theorem~\cite{Brieskorn}. If the K\"ahler assumption in Brieskorn's rigidity theorem for the complex structure of $Q$ 
could be dropped, then, together with Theorem~\ref{Qproj}, it would imply that $S^6$ cannot have a complex structure.

\subsection{Chern number calculations}

By the general theory of Borel and Hirzebruch~\cite{BH}, the homogeneous spaces $Z$ and $Q$ carry $2$, respectively
$4$, invariant almost complex structures, up to conjugation\footnote{This follows from Schur's lemma and the fact that the number of 
irreducible summands in the isotropy representation is $2$, respectively $3$.}.

On the manifold $Z$, the second almost complex structure, apart from the integrable and K\"ahler structure of the twistor space,
corresponds to the Eells--Salamon construction~\cite{ES} performed on the twistor fibration $\pi_Z\colon Z\longrightarrow M$.
The two structures are conjugate along the complex fibers of this fibration, while agreeing on a suitable complement.
It is known that the second, non-integrable, structure is nearly K\"ahler, and so we denote it by $N$, although it lives on the manifold $Z$,
which however is considered with its integrable and K\"ahler structure. The proof of Theorem~\ref{Zrigid} effectively tells us 
what the Chern classes of $Z$ are, and this in turn can be used to work out the Chern classes of $N$ as well. This leads to 
the values for the Chern numbers of $Z$ and $N$ given in Table~\ref{tab:Znumbers}.

\begin{table}[ht!]\centering
	\begin{tabular}{ccc} \toprule
					& $Z$ 		& $N$ \\ \midrule
		$c_5$ 		& $6$		& $6$ \\
		$c_1^5$ 	& $4374$	& $-18$\\
		$c_1^3c_2$	& $2106$	& $-6$\\
		$c_1^2c_3$	& $594$		& $18$\\
		$c_1c_4$	& $90$		& $18$\\
		$c_1c_2^2$	& $1014$	& $-2$\\
		$c_2c_3$	& $286$		& $6$\\ \bottomrule\\
	\end{tabular}
	\caption{Chern numbers of invariant almost complex structures on $G_2/U(2)_+$.}
	\label{tab:Znumbers}
\end{table}


For the quadric $Q\subset \C P^6$ with its complex structure the Chern classes and Chern numbers can be easily computed by the adjunction 
formula. Considering $S^6$ as an almost complex manifold with its $G_2$-invariant almost complex structure gives $TS^6$
the structure of a complex vector bundle. Its projectivization $\PP (TS^6)$ and the projectivization $\PP (T^*S^6)$ of its dual 
account for two more invariant almost complex structures on the manifold $Q$. The fourth invariant almost complex structure $X$
predicted by the theory of Borel--Hirzebruch~\cite{BH} is more mysterious, but is related to $Q$, respectively to $\PP (TS^6)$, by versions 
of the Eells--Salamon construction~\cite{ES} performed on $p$, respectively on $\pi_Q$, see Figure~\ref{fig:Qfliprelations} in Section~\ref{s:quadric}. 
Again this allows us to compute all the Chern numbers, leading to the numbers in Table~\ref{tab:Qnumbers}.

\begin{table}[ht!]\centering
	\begin{tabular}{ccccc} \toprule
					& $Q$		& $\PP(TS^6)$ 	& $\PP(T^*S^6)$	& $X$\\ \midrule
		$c_5$ 		& $6$ 		& $6$ 			& $6$			& $6$\\
		$c_1^5$ 	& $6250$	& $-486$		& $486$			& $-2$\\
		$c_1^3c_2$	& $2750$ 	& $-162$		& $162$			& $2$\\
		$c_1^2c_3$	& $650$ 	& $18$ 			& $18$			& $2$\\
		$c_1c_4$	& $90$ 		& $18$ 			& $18$			& $-6$\\
		$c_1c_2^2$	& $1210$ 	& $-54$ 		& $54$			& $-2$\\
		$c_2c_3$	& $286$ 	& $6$			& $6$			& $-2$\\ \bottomrule \\
	\end{tabular}
	\caption{Chern numbers of the invariant almost complex structures on $G_2/U(2)_-$.}
	\label{tab:Qnumbers}
\end{table}

\subsection{Outline}
In Section~\ref{s:twistor} we prove Theorem~\ref{Zrigid} and we carry out the Chern number calculations for the invariant 
almost complex structures on $Z$. In Section~\ref{s:quadric} we explain how the invariant almost complex structures on $Q$
are related to each other, and how two of them come from the projectivised complex tangent and cotangent bundles of $S^6$.
This leads to the proof of Theorem~\ref{Qproj} and the calculations of all the Chern numbers. In Section~\ref{s:full} we 
explain why we do not deal in detail with the full flag manifold $G/T^2$ here, and in Section~\ref{s:comparison} we 
compare our Chern number calculations to other calculations in the literature. In particular, we correct several errors 
in previous calculations.

\subsection*{ Acknowledgement:}
We are grateful to R.~Coelho, M.~Hamilton and U.~Semmelmann for helpful discussions.

\section{The twistor space}\label{s:twistor}

In this section we prove Theorem~\ref{Zrigid} and we carry out the calculations of Chern numbers summarised in Table~\ref{tab:Znumbers}.
To do so we need to know the cohomology ring of the twistor space $Z$, determined by combining Borel's thesis and the work of
Borel--Hirzebruch~\cite{BH}, Toda~\cite{Toda} and Ishitoya--Toda~\cite{IT}. The final result can be summarised as follows.
\begin{prop}
The integral cohomology groups of $Z$ agree with those of $\C P^5$. If $L\in H^2(Z,\Z)$ is a generator, then 
$$
\frac{1}{3}L^2 \ , \ \ \frac{1}{6}L^3 \ , \ \ \frac{1}{18}L^4 \ , \ \ \frac{1}{18}L^5 
$$
are integral generators of the higher-degree cohomology groups.
\end{prop}
Once one has understood the simple cohomology structure of the Wolf space $M=G_2/SO(4)$, most of the 
calculation for $Z$ can be carried out using the Gysin sequence for the twistor fibration $\pi_Z\colon Z\longrightarrow M$.
This gives the additive information about the cohomology of $Z$ and it shows that the square of a generator in degree $4$ is twice 
a generator in degree $8$. Together with Poincar\'e duality, this reduces the determination of the constants appearing 
in the Proposition to the determination of a single number, e.g.~the statement that a generator in top degree is  $\frac{1}{18}L^5$.
While this can be obtained purely from algebraic topology, it is also known from the point of view of 
complex geometry. For example, the fact that $Z$ has Fano genus $=10$ (see Mukai~\cite[p.~3000]{Mukai}) exactly 
means that $L^5$ evaluates as $\pm 18$ on the fundamental class of $Z$.
Alternatively, one can exploit the fact that $Z$ is the twistor space of the quaternionic K\"ahler manifold $M=G_2/SO(4)$ of 
positive scalar curvature. This implies that $\frac{1}{2}\langle L^5 , [Z] \rangle +5$ is the dimension of the isometry group of $M$, 
by a result of Poon and Salamon~\cite[Theorem 2.2 (ii)]{PS}. Since the dimension of $G_2$ is $14$, this 
implies that $\langle L^5 , [Z] \rangle = 18$, as claimed.

\begin{lem}\label{Pont}
The Pontryagin classes of $Z$ are $p_1(Z)=\frac{1}{3}L^2$ and $p_2(Z)=\frac{1}{9}L^4$.
\end{lem}
This follows from the computation of the Pontryagin classes of $M$ by Borel and Hirzebruch~\cite{BH}, 
together with $TZ=T\pi_Z\oplus\pi_Z^*TM$ and the description of 
the pullback in cohomology for the twistor fibration $\pi_Z\colon Z\lra M$.

We can now prove Theorem~\ref{Zrigid}.
\begin{proof}[Proof of Theorem~\ref{Zrigid}]
The strategy of the proof is to first determine the first Chern class of any K\"ahlerian complex manifold homeomorphic to $Z$. 
Since the second Betti number of $Z$ is one, the K\"ahler class may be taken to be integral, and so the structure is in 
fact projective by the Kodaira embedding theorem.

Since all the Betti numbers are $0$ or $1$, all the Hodge numbers $h^{p,q}$ vanish for $p\neq q$. 
Therefore those Chern numbers which are determined by the Hodge numbers take the same values on $Z$ as on $\C P^5$.
This applies in particular to $c_1c_4$ by a result of Libgober and Wood~\cite{LW} (compare also~\cite{Sal}) and so
$c_1c_4=90$. Therefore, $c_1$ cannot be zero, and its divisibility divides $90$. Moreover, since $Z$ is 
not spin, the divisibility of $c_1$ is odd. 

We write $c_1=d L$, with $L$ the positive integral generator of the second cohomology.
For $d>0$ the complex structure is Fano, whereas for $d<0$ it has ample canonical bundle.

For Fano manifolds Kobayashi and Ochiai~\cite{KO} proved that the divisibility of $c_1(Z)$, known in this case 
as the Fano index, is at most $1+\dim_{\C}(Z)$, and if it equals $\dim_{\C}(Z)$, then the Fano manifold is isomorphic 
to the quadric. In our case, since $Z$ has a different cohomology ring from $Q$, this means $k<5$.
We conclude that  $d\in\{\pm 1,\pm 3,-5,-9,-15,-45\}$.
	
	Since the cohomology is torsion-free, the {\it integral} Pontryagin classes are homeomorphism invariants, 
	and are as given in Lemma~\ref{Pont}. Expressing the Pontryagin classes in terms of Chern classes, we have:
	\begin{equation}\label{p1}
		p_1=c_1^2-2c_2
		\end{equation}
		\begin{equation}\label{p2}
		p_2=c_2^2-2c_1c_3+2c_4 \ .
	\end{equation}
	The Hirzebruch--Riemann--Roch theorem for the Todd (or arithmetic) genus
	yields another constraint on the Chern classes:
	\begin{equation*}
		1=\frac{1}{1440}\big(-c_1^3c_2+c_1^2c_3+3c_1c_2^2-c_1c_4\big) \ .
	\end{equation*}
	Plugging in $c_1c_4=90$, we find:
	\begin{equation}\label{111}
		c_1^2c_3=1530+c_1^3c_2-3c_1c_2^2 \ .
	\end{equation}
Together with the Pontryagin classes~\eqref{p1}, \eqref{p2}, this relation suffices to rule out all possible values except $d=3$, 
as we will now show. 
	
First, assume $d=\pm 1$. Then $c_2=\frac{1}{3}L^2$, hence $c_1^2c_3=1530\pm 4$ while at the same time
	\begin{equation*}
		c_1^2c_3=\frac{1}{2}\big(c_1c_2^2+2c_1c_4-c_1p_2\big)=90 \ .
	\end{equation*}
This is a contradiction. If $d$ is a multiple of nine, then~\eqref{111} gives $0\equiv 1530\mod 27$, which is also a contradiction. 
For $d=-15$, we have $c_4=-6\cdot \frac{1}{18}L^4$ and the expression~\eqref{p1} for $p_1$ yields $c_2=337\cdot\frac{1}{3}L^2$. 
But then the expression~\eqref{p2} for $p_2$ shows that $c_1c_3=113562\cdot \frac{1}{18}L^4$, which is not divisible by $15$ 
and therefore contradictory. 
	
Now assume $d=-5$. Then $c_2=37 \cdot\frac{1}{3}L^2$ and we find $c_1c_3=1350\cdot\frac{1}{18}L^4$, which implies that $c_1^2c_3=-6750 $. On the other hand, $c_1^2c_3>c_1^3c_2-3c_1c_2^2=13320$, ruling out this possibility. Finally, if $d=-3$ we find $c_2=13 \cdot\frac{1}{3}L^2$ and $c_4=-30 \cdot\frac{1}{18}L^4$. The two expressions for $c_1^2c_3$ then yield the values $-411$ and $2286$. This leaves only the possibility that $d=3$.
	
Now we have established that our K\"ahler manifold is Fano of index three. 
Its Fano \emph{coindex} $\dim_{\C} Z+1-3$ also equals three, and thus we may appeal to the classification of Fano manifolds with 
coindex three, due to Mukai~\cite{Mukai}; cf.~also~\cite[Theorem~7]{AC}. 
Under a technical assumption which was later verified by Mella~\cite{Mella}, Mukai~\cite[Prop.~1]{Mukai} proved that this manifold
is what he calls an \emph{$F$-manifold of the first species} with \emph{Fano genus} $g=\frac{1}{2}L^5+1=10$. In Theorem~2 of the 
same paper, he established that this manifold is biholomorphic to the twistor space $Z$, equipped with its canonical complex structure 
(see also Remark~1 in {\it loc.~cit.}). This completes our proof.
\end{proof}

The arguments in the above proof tell us all the Chern classes of the twistor space $Z$. It has 
$c_1(Z)=3L$, and $c_2(Z)=13\cdot\frac{1}{3}L^2$. Since $c_1c_4(Z)=90$, we must have 
$c_4(Z)=30\cdot\frac{1}{18}L^4$. Now using the formula for $p_2(Z)$, one finds $c_3(Z)=22\cdot\frac{1}{6}L^3$.
Multiplying out and evaluating, one finds the Chern numbers of $Z$ given in the first column of 
Table~\ref{tab:Znumbers}. 

As we mentioned in Section~\ref{s:intro}, the second invariant almost complex structure $N$ on the twistor space is obtained 
from its K\"ahler structure by conjugating along the fibers of the twistor fibration. 
This description allows us to compute its Chern classes, starting from those of $Z$:
\begin{prop}
	The total Chern class of the nearly K\"ahler structure $N$ is
	$$
	c(N)=\frac{1-L}{1+L}c(Z)=1+L+\frac{1}{3}L^2-L^3-L^4-\frac{1}{3}L^5 \ .
	$$
\end{prop}
\begin{proof}
	Denoting the subbundle of $TZ$ given by tangent vectors along the fibers by $T\pi_Z$, the orthogonal complement $D$ of 
	$T\pi_Z$ with respect to the invariant K\"ahler--Einstein metric is a holomorphic contact structure on $Z$, see~\cite{S}. 
	We now have a decomposition $TZ=T\pi_Z\oplus D$ and, by the Eells--Salamon construction~\cite{ES}, $TN=(T\pi_Z)^{-1}\oplus D$. 
	A theorem of Kobayashi~\cite{Kobayashi} implies that $c_1(Z)=3c_1(T\pi_Z)$, and since $c_1(Z)=3L$, we conclude that $c(T\pi_Z)=1+L$. 
	This means that $c(N)=\frac{1-L}{1+L}c(Z)$, and multiplying this out one obtains the claimed formula.
\end{proof}
Keeping in mind that the orientation induced by the almost complex structure of $N$ is opposite to that of $Z$, 
it is now straightforward to compute the Chern numbers, to obtain the second column of Table~\ref{tab:Znumbers}.

\begin{rem}
The twistor space is actually a $3$-symmetric space in the sense of Gray and Wolf~\cite{W-G}, and therefore~\cite{G} carries a nearly
K\"ahler structure induced by the $3$-symmetric structure. The almost complex manifold underlying this nearly K\"ahler
structure is the $N$ considered above, cf.~\cite{AGI,CMH}
\end{rem}

\newpage 

\section{The quadric}\label{s:quadric}

In this section we calculate the Chern numbers displayed in Table~\ref{tab:Qnumbers} and we prove Theorem~\ref{Qproj}.

First, we have an easy consequence of obstruction theory.
\begin{lem}\label{lem:S6}
The sphere $S^6$ has a unique homotopy class of almost complex structures.
\end{lem}

This leads to the following descriptions of the smooth manifold underlying the five-dimensional complex quadric.
\begin{prop}\label{propQuadric}
The following ten-dimensional manifolds are all diffeomorphic to each other:
\begin{enumerate}
\item the quotient $G_2/U(2)_-$ from Figure~\ref{diag},
\item the Grassmannian $Gr_2(\R^7)$ of oriented $2$-planes in $\R^7$,
\item the complex quadric $Q\subset\C P^6$, and
\item the projectivized complex tangent and cotangent bundles $\PP (TS^6)$ and $\PP (T^*S^6)$ for any almost 
complex structure in $S^6$.
\end{enumerate}
\end{prop}
The diffeomorphism between (1) and (3) in the Proposition is compatible with the complex structure in the sense that the standard
complex structure of the quadric $Q$ is $G_2$-invariant, and therefore accounts for the unique, up to conjugation, $G_2$-invariant
integrable almost complex structure predicted by Borel and Hirzebruch~\cite{BH}. By Brieskorn's theorem~\cite{Brieskorn} this 
is the only K\"ahlerian structure on this manifold.
\begin{proof}
By the Lemma, the projectivized tangent and cotangent bundles in (4) do not depend on the choice of almost 
complex structure. Moreover, for any complex vector bundle $E$, a choice of Hermitian metric induces a 
diffeomorphism between $\PP (E)$ and $\PP (E^*)$.

So we have a unique manifold in (4), and we choose to represent it using the standard $G_2$-invariant 
almost complex structure of $S^6=G_2/SU(3)$. It then follows that $\PP (TS^6)$ is also homogeneous under $G_2$,
and must be of the form $G_2/U(2)$ with $U(2)\subset SU(3)$. This shows that we have $G_2/U(2)_-$, and not $G_2/U(2)_+$;
compare Figure~\ref{diag}. This gives the diffeomorphism between (1) and (4).

The Grassmannian in (2) is usually written as the symmetric space $SO(7)/SO(5)SO(2)$, but it is well known that the $SO(7)$-action restricts to a 
transitive action of $G_2\subset SO(7)$ with isotropy $U(2)$, and this gives the diffeomorphism between (1) and (2); cf.~Kerr~\cite[p.~162]{Kerr}.

The identification between (2) and (3) is well known, see for example~\cite{Brieskorn,Bryant}. A diffeomorphism between 
(2) and (4) is described explicitly by Bryant~\cite[p.~200]{Bryant}.
\end{proof}

\begin{rem}
As pointed out to us by the referee, an interpretation of $Q$ as a twistor space of $S^6$ appears in the paper of O'Brian and Rawnsley~\cite{OR}.
\end{rem}

Determining the Chern classes of the quadric $Q$ is a routine exercise, using adjunction for $\iota\colon Q\hookrightarrow \C P^6$.
The total Chern class $c(\C P^n)$ is given by $(1+H)^{n+1}$, so that the Whitney product formula yields
\begin{equation*}
	(1+\iota^*H)^7=c(Q)(1+2\iota^*H)
\end{equation*}
Matching terms degree by degree yields:
\begin{lem}\label{lemquadric}
	The total Chern class of the quadric $Q$ is given by
	\begin{equation*}
		c(Q)=1+5h+11h^2+13h^3+9h^4+3h^5 \ , 
	\end{equation*}
	where $h=\iota^*H$ is a primitive generator of $H^2(Q;\Z)$.
\end{lem}

To obtain the Chern numbers, the only subtle point one has to keep in mind is that the fundamental class $[Q]\in H_{10}(Q;\Z)$ maps 
to {\it twice} the generator of $H_{10}(\C P^6;\Z)$ under $\iota_*$, since $Q$ is a quadric. The resulting Chern numbers are listed in 
the first column of Table~\ref{tab:Qnumbers}.

The second column of that table is a direct corollary of the next Proposition. 
Recall that by Lemma~\ref{lem:S6} the almost complex manifold $\PP(TS^6)$ is independent, up to homotopy of 
almost complex structures, of the chosen almost complex structure of $S^6$.
\begin{prop}\label{prop:cohomofPTS6}
	The integral cohomology ring of $\PP(TS^6)$ is generated by two elements, $x\in H^6(\PP(TS^6))$ and $y\in H^2(\PP(TS^6))$, which satisfy the relations
	\begin{equation*}
		x^2=0\qquad \qquad y^3=-2x \ .
	\end{equation*}
The total Chern class is given by 
\begin{equation*}
		c(\PP(TS^6))=1+3y+3y^2+2x+6xy+6xy^2 \ .
	\end{equation*}
\end{prop}
\begin{proof}
	Let $\alpha\in H^6(S^6;\Z)$ be the orientation class. Then $c_3(S^6)=2\alpha$ since the Euler characteristic of $S^6$ equals $2$. Since $S^6$ has no non-trivial cohomology in any other (positive) degree, it generates the entire cohomology ring. 
	
	Now set $x=p^*\alpha$, where $p\colon\PP(TS^6)\lra S^6$ is the projection. Then clearly $x^2=0$ for dimension reasons, while Grothendieck's definition of Chern classes shows that $y^3+2x=0$, where $y$ is the hyperplane class of $\PP(TS^6)$. The Leray--Hirsch theorem now tells us us that these are the only relations. 
	Finally, note that $xy^2$ is the positive generator of the cohomology of top degree, since $\alpha$ and $y$ are positive generators on the base and fiber.

	We employ the fibration $p\colon\PP(TS^6)\lra S^6$ and decompose the tangent bundle as $T\PP(TS^6)=Tp\oplus p^*TS^6$, 
	where $Tp$ denotes the subbundle formed by tangent vectors along the fiber. Clearly $p^*c(S^6)=1+2x$, so all that is left is to determine is $c(Tp)$. 
	Let $H$ denote the dual of the tautological line bundle over the projectivization. Then we have the relative Euler sequence 	
	\begin{equation*}
		\begin{tikzcd}
			0 \ar[r] & H^{-1} \ar[r] & p^*TS^6 \ar[r] & H^{-1}\otimes Tp \ar[r] & 0 \ .
		\end{tikzcd}
	\end{equation*}
	This implies that $p^*TS^6\cong H^{-1}\oplus (H^{-1}\otimes Tp)$ as complex vector bundles. 
	Twisting by $H$, we find $H\otimes p^*TS^6\cong \C\oplus Tp$. Thus, we see that
	\begin{equation*}
		c(Tp)=c(H\otimes p^*TS^6) \ .
	\end{equation*}
		Now $c(H)=1+y$ shows that $c(Tp)=1+3y+3y^2$. Now, we apply the Whitney product formula and find
	\begin{equation*}
		c(\PP(TS^6))=(1+3y+3y^2)(1+2x)=1+3y+3y^2+2x+6xy+6xy^2 \ ,
	\end{equation*}
	which was our claim.
\end{proof}
For concreteness and for easy comparison with other results, we carry out one of the calculations of Chern numbers explicitly.
\begin{ex}\label{ex}
According to Proposition~\ref{prop:cohomofPTS6}, the almost complex manifold $\PP (TS^6)$ has $c_1=3y$ and $c_3=2x$. This gives 
$c_1^2c_3=18xy^2$, and since $xy^2$ is the positive generator in top degree, $c_1^2c_3$ evaluates to give $18$ on the fundamental 
class induced by the orientation coming from the almost complex structure.
\end{ex}

In exactly the same way as for $\PP (TS^6)$, one can compute the Chern numbers for $\PP (T^*S^6)$. 
\begin{prop}\label{prop:cohomofPT*S6}
	The integral cohomology ring of $\PP(T^*S^6)$ is generated by two elements, 
	$x\in H^6(\PP(T^*S^6))$ and $y\in H^2(\PP(T^*S^6))$, which satisfy the relations
	\begin{equation*}
		x^2=0\qquad \qquad y^3=2x \ .
	\end{equation*}
The total Chern class is given by 
\begin{equation*}
		c(\PP(T^*S^6))=1+3y+3y^2+2x+6xy+6xy^2 \ .
	\end{equation*}
\end{prop}
This looks formally exactly like Proposition~\ref{prop:cohomofPTS6}, with the only difference that now $y^3=2x$ instead of $y^3=-2x$.
This leads to a sign change in some Chern numbers, but not in others. The result is given in the third column of Table~\ref{tab:Qnumbers}.

We can now prove Theorem~\ref{Qproj}.
\begin{proof}[Proof of Theorem~\ref{Qproj}]
If $S^6$ admits a complex structure, then projectivizing the holomorphic tangent and cotangent bundles gives two complex manifolds 
denoted $\PP (TS^6)$ and $\PP (T^*S^6)$. Like all projectivized cotangent bundles, the latter carries a tautological holomorphic contact structure. 
By Proposition~\ref{propQuadric} these complex manifolds are diffeomorphic to each other, and to $Q$. The two complex structures
cannot be equivalent because their Chern numbers do not agree, as seen by inspecting Table~\ref{tab:Qnumbers}.
\end{proof}

It remains to discuss the fourth invariant almost complex structure on the manifold $Q$, which will turn out to be distinct from $Q$ and 
from $\PP (TS^6)$ and $\PP (T^*S^6)$. Recall that the almost complex structures on $\PP (TS^6)$ and on $\PP (T^*S^6)$ are related 
by conjugation on the fiber of the fibration $p\colon Q\lra S^6$, 
which is the precise analog of the Eells--Salamon construction by which we related $Z$ and $N$ in the previous section. 
Now, since the tangents to the fibers of $p$ form a complex subbundle for the integrable complex structure of $Q$ as well, we can 
perform this conjugation construction on $Q$ to get the missing invariant almost complex structure on this homogeneous space.

We consider a decomposition $TQ\cong Tp\oplus D$, where $D$ is a complementary complex subbundle.
 Recall that $c(Q)=1+5h+11h^2+13h^3+9h^4+3h^5$, and that $h$ restricts to the hyperplane class on each fiber, which is just a copy of $\C P^2$. 
 Thus $c_1(Tp)=3h$, which forces $c_1(D)=2h$. Similarly, we find $c_2(Tp)=3h^2$ and $c_2(D)=2h^2$. Since $Tp$ has rank two, we see that 
 $c_3(D)=h^3$ and $c(Q)$ factorizes as 
 $$
 c(Q)=(1+3h+3h^2)(1+2h+2h^2+h^3) \ .
 $$ 
Now we conjugate on the fiber, replacing $Tp$ by its conjugate $\overline{Tp}$. The resulting almost complex manifold will be denoted by $X$, 
and its tangent bundle has (by definition) a decomposition $TX\cong \overline{Tp}\oplus D$. The following is then obvious:
\begin{prop}
	The almost complex structure $X$ has total Chern class 
	\begin{equation*}
		c(X)=c(Q)\frac{1-3h+3h^2}{1+3h+3h^2} 		=1-h-h^2+h^3+3h^4+3h^5 \ .
	\end{equation*}
\end{prop}
Note that this flip does not change the orientation, since $Tp$ is a rank two subbundle. Therefore, $xy^2$ remains the positive generator 
of the cohomology in top degree. It is already clear from the expression for the Chern class that the Chern numbers $X$ will be drastically 
different than those of $Q$, $\PP(TS^6)$ and $\PP(T^*S^6)$. They are shown in the last column of Table~\ref{tab:Qnumbers}. 

A non-trivial consistency check for these calculations is provided by observing that, on the one hand, $Q$ and $\PP(T^*S^6)$, and, on the 
other hand, $X$ and $\PP(TS^6)$ are related by conjugation along the fiber of $\pi_Q$, leading to the diagram in Figure~\ref{fig:Qfliprelations}.
We computed for $Q$, $\PP(TS^6)$ and $\PP(T^*S^6)$ from first principles, and then used the top horizontal conjugation to do the 
calculation for $X$. The vertical conjugation on the right gives the same result for $X$, and the vertical conjugation on the left shows that the 
two calculations for $Q$ and $\PP(T^*S^6)$ are consistent.

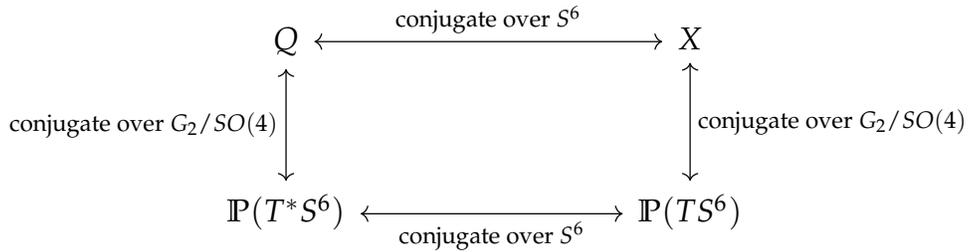
\begin{figure}[ht!]
	\begin{equation*}
		\begin{tikzcd}[column sep=3.5cm,row sep=1.5cm]
			Q \ar[r,leftrightarrow,"\text{conjugate over }S^6"] 
			\ar[d,leftrightarrow,"\text{conjugate over }G_2/SO(4)"']
			& X \ar[d,leftrightarrow,"\text{conjugate over }G_2/SO(4)"]\\
			\PP (T^*S^6) \ar[r,leftrightarrow,"\text{conjugate over }S^6"']
			& \PP (TS^6)
		\end{tikzcd}
	\end{equation*}
	\caption{Conjugation on isotropy summands for $G_2/U(2)_-$.}\label{fig:Qfliprelations}
\end{figure}

\section{The full flag manifold}\label{s:full}

In this section we explain why we focussed on $Q$ and $Z$ in this paper, and are not proving any results for the full flag manifold $G_2/T^2$.

The isotropy representation of $G_2/T^2$ splits into $6$ complex one-dimensional irreducible summands. According to 
Borel--Hirzebruch~\cite{BH} this means that there are $2^6$ invariant almost complex structures. Up to an overall conjugation there are 
still $2^5=32$ structures, only one of which is integrable and K\"ahlerian. It is possible to carry out Chern number calculations for 
all these structures through Lie theory, as is done in~\cite{BH,CMH,Grama} for many other cases. Grama, Negreiros and Oliveira~\cite[Subsection~8.4.1]{Grama}
give the Chern numbers for the unique integrable structure, but not for the other ones. We have not tried to do these calculations 
systematically, because instead of ``digging roots and lifting weights'', we want to calculate geometrically, and for most of the 
non-integrable structures there is no convenient geometric description. 

As for the rigidity results for K\"ahlerian complex structures, Brieskorn's theorem~\cite{Brieskorn} for $Q$ and Theorem~\ref{Zrigid}
of this paper for $Z$, there cannot be such a result for $G_2/T^2$, as we now explain.

The second Betti number of $G_2/T^2$ is $2$, and so for any K\"ahlerian complex structure the whole second cohomology is of type 
$(1,1)$, and therefore, by the Kodaira embedding theorem, the structure is actually complex algebraic. The fibrations of $G_2/T^2$,
equipped with its invariant K\"ahler structure, over $Q$ and over $Z$ in Figure~\ref{diag} are holomorphic $\C P^1$-bundles, and
the most one could hope to prove in the direction of a rigidity theorem would be that any other K\"ahlerian complex structure is 
also such a $\C P^1$-bundle, and perhaps a deformation of the standard one.

Consider the analogous situation for the $3$-dimensional flag manifold 
$$
F(1,2)=U(3)/U(1)\times U(1)\times U(1) \ ,
$$ 
where the standard invariant complex structure is that of the projectivized tangent or cotangent bundle\footnote{This is an exceptional case,
in which the projectivizations of the tangent and of the cotangent bundles are biholomorphic.} of $\C P^2$, compare~\cite{CMH}.
Using the methods of~\cite{CK} one can show that any K\"ahlerian complex structure on $F(1,2)$ is the projectivisation of a 
holomorphic rank $2$ bundle over $\C P^2$, whose underlying smooth bundle is isomorphic to the tangent bundle of $\C P^2$.
Although the stable holomorphic structure on this vector bundle is unique, see e.g.~\cite{Kot}, there are lots of other, unstable, holomorphic 
structures~\cite{Sch}, whose projectivisations give non-standard K\"ahlerian structures on the smooth manifold $F(1,2)$. 
{\it Mutatis mutandis} one can find such non-standard complex structures on $G_2/T^2$.

\section{Comparison with other calculations}\label{s:comparison}

In this section we compare our calculations of the Chern numbers with results already contained in the 
literature.

\subsection{}\label{BE}
The homogeneous spaces $G_2/U(2)_{\pm}$ are discussed as examples in the book of Baston and Eastwood~\cite{BE}.
They are first mentioned in Example (6.2.8), where it is remarked that they are topologically distinct, and that one of them is the 
five-dimensional complex quadric. In Example (6.3.4) the Chern and Pontryagin classes are written down in terms of roots and weights.
The conclusion is that the first Pontryagin class of $Z$ is a generator of $H^4(Z;\Z)$, which checks with our calculation in Section~\ref{s:twistor}. 
The same conclusion is stated for $Q$ at the 
top of~\cite[p.~61]{BE}, but this is clearly a misprint, since the authors write that their calculation is consistent with the identification
of this homogeneous space with the quadric, for which the first Pontryagin class is $3$ times a generator of $H^4(Q;\Z)$, as can be 
seen from Lemma~\ref{lemquadric} above.

\subsection{}
Our calculations of the Chern numbers of $Z$ can be compared with the work of Poon--Salamon~\cite{PS} and
of Semmelmann--Weingart~\cite{SW}.
The generator  $L$ of the second cohomology of the twistor space is the first Chern class of an ample line bundle, because
$Z$ is Fano. Thus one can consider $(Z,L)$ as a polarised projective algebraic variety with Hilbert polynomial
$$
P(r)=\chi (Z,\OO (L^r))=\sum_{i=0}^{5}(-1)^i\dim_{\C}H^i(Z,\OO (L^r)) \ .
$$
By the Hirzebruch--Riemann--Roch theorem, this can be calculated as
$$
P(r) = \langle ch(L^r)Todd(Z),[Z]\rangle \ ,
$$
which is a polynomial of degree (at most) $5$ in $r$.  Let us just write out the terms
of highest degree in $r$:
\begin{equation*}
P(r) = \frac{1}{5! \cdot 3^{5}}c_1(Z)^{5}r^{5}
+ \frac{1}{2 \cdot 4! \cdot 3^{4}}c_1(Z)^{5}r^{4}
+ \frac{1}{12 \cdot 3! \cdot 3^{3}}(c_1(Z)^{5}+c_1(Z)^{3}c_2(Z))r^{3}+\ldots
\end{equation*}
Now Poon--Salamon~\cite[Thm.~2.2~(iii)]{PS} Semmelmann--Weingart~\cite[p.~159]{SW} 
have calculated this Hilbert polynomial completely, and obtained:
\begin{equation}\label{eq:SW}
P(r) = \frac{1}{120}(r+2)(3r+5)(2r+3)(3r+4)(r+1) \ .
\end{equation}
Expanding this in powers of $r$ we find:
$$
P(r)= \frac{3}{20}r^5 + \frac{9}{8}r^4+\frac{10}{3}r^3+\frac{39}{8}r^2+\frac{211}{60}r+1 \ .
$$
Comparing the coefficients of $r^{5}$ in the two expansions, we find $c_1^5(Z)=18\cdot 3^5=4373$, 
which checks with what we computed in Section~\ref{s:twistor}.

One can determine further combinations of Chern numbers for $Z$ by looking at the terms of lower order in $r$. The
coefficients of $r^{4}$ give no new information, but provide a consistency check for the calculation of
$c_1^{5}(Z)$. Combining this calculation with the comparison of the coefficients of $r^{3}$, we find $c_1^3c_2(Z)=2106$, which 
again checks with what we computed in Section~\ref{s:twistor}.
One could calculate some more Chern numbers by looking at the further terms in the expansions, but this would not 
be enough to compute all the Chern numbers of $Z$.

As we have computed all the Chern numbers  of $Z$ independently, we obtain a new proof of the formula~\eqref{eq:SW}
for the Hilbert polynomial first proved in~\cite{PS,SW}.

\subsection{}
Hirzebruch~\cite{H05} compared the Chern numbers of the projectivizations $\PP (TB)$ and $\PP (T^*B)$ for 
arbitrary complex $3$-folds $B$, and, of course, his calculations apply equally well when $B$ is only almost complex.
In the case where $B$ has vanishing first Chern class, Hirzebruch gave complete formulas for all the Chern numbers
of $\PP (TB)$ and $\PP (T^*B)$ expressed as universal multiples of the Euler characteristic $c_3(B)$, see~\cite[Table~(5)]{H05}.
Using $c_3(S^6)=2$, his calculation gives the values we have displayed in the middle two columns of Table~\ref{tab:Qnumbers}.

\subsection{}
The homogeneous spaces $G_2/U(2)_{\pm}$ also appear in the work of Araujo and Castravet, see~\cite[Subsection~6.4]{AC},
where they are denoted $G/P_1$ and $G/P_2$, because the two copies of $U(2)$ are maximal parabolic subgroups.
The space $G/P_1$ is the five-dimensional quadric $Q$, and $G/P_2$ is identified as a Mukai variety of genus $10$, in other 
words, $G/P_2$ is the twistor space $Z$. Araujo and Castravet~\cite{AC} claimed that the degree two part of the Chern character 
of $G/P_2$ is given by 
\begin{equation}\label{eq:AC}
ch_2(G/P_2)=\frac{1}{2}H^2 \ ,
\end{equation}
where $H$ is the ample generator of the Picard group. In our notation of Section~\ref{s:twistor}, $L$ can be taken to be $H$, and 
our calculations of the first two Chern classes of $Z$ give 
$$
ch_2(Z)=\frac{1}{2}(c_1^2(Z)-2c_2(Z))= \frac{1}{2}((3L)^2-2\cdot 13\cdot\frac{1}{3}L^2)=\frac{1}{6}L^2 \  ,
$$
showing that~\eqref{eq:AC} is not correct.

Note that in general $ch_2$ is one half the first Pontryagin class, so these calculations can be compared with the 
discussion in Subsection~\ref{BE} above.

\subsection{}
Recently, Grama, Negreiros and Oliveira~\cite{Grama} carried out Chern number calculations for all the invariant 
almost complex structures on $G_2/U(2)_{\pm}$ via Lie theory, see~\cite[Subsection~8.4]{Grama}. Their Table~10
corresponds to our Table~\ref{tab:Qnumbers}. The Chern numbers for the integrable complex K\"ahler structure of $Q$
given in~\cite[Table~10]{Grama} agree with ours, up to an overall sign change. Note that the Euler characteristic of $Q$ 
is $+6$, so it is clear that $c_5$ must be $+6$, and not $-6$. However, the numbers for the non-integrable almost
complex structures given in~\cite[Table~10]{Grama} are off in more ways than just by a sign. For instance,
we computed in Example~\ref{ex} above that $c_1^2c_3(\PP (TS^6))=18$, which also follows from~\cite[Table~(5)]{H05}. 
The values for $c_1^2c_3$ for non-integrable 
structures appearing in~\cite[Table~10]{Grama} are $-9$ and $-2$. Note, by the way, that the columns of 
our Table~\ref{tab:Qnumbers} should exactly match the columns of~\cite[Table~10]{Grama}, perhaps up to interchanging 
the two middle columns. The same remarks apply to~\cite[Table~11]{Grama}, which corresponds to our
Table~\ref{tab:Znumbers}. The Chern numbers for the integrable complex structure $Z$ agree, but for the non-integrable
$N$, our values and those in~\cite{Grama} are quite different, not just up to an overall sign.

In~\cite[Proposition~8.10]{Grama} the authors state that $G_2/U(2)_-$ has at least $3$ distinct invariant almost complex structures. 
In fact, it has exactly $4$ invariant almost complex structures by Borel--Hirzebruch~\cite{BH}, and all four are distinct because 
of our Table~\ref{tab:Qnumbers}. 


\bibliographystyle{amsplain}

\bigskip

\end{document}